\documentclass[11pt,leqno]{amsart}
\usepackage[utf8]{inputenc}
\usepackage{graphicx}
\usepackage[english]{babel}
\usepackage{amsmath,amssymb, hyperref}
\usepackage{enumitem}
\usepackage{mathtools}
\usepackage{xcolor}
\usepackage{thm-restate}

\usepackage{comment}
\usepackage{todonotes}

\newtheorem{theorem}{Theorem}[section]
\newtheorem{lemma}[theorem]{Lemma}
\newtheorem{proposition}[theorem]{Proposition}
\newtheorem{corollary}[theorem]{Corollary}
\newtheorem{question}{Question}
\theoremstyle{definition}
\newtheorem{definition}{Definition}[section]
\newtheorem{example}[definition]{Example}
\newtheorem{remark}[definition]{Remark}

\def\N{\mathbb{N}}
\def\Z{\mathbb{Z}}
\def\R{\mathbb{R}}

\renewcommand{\kappa }{\varkappa}
\newcommand\e{\varepsilon}

\renewcommand{\d}{{\rm d}}

\newcommand{\cS }{\mathcal S}

\begin{document}

\title{Concise formulae in groups of non-positive curvature}
\author[L. Ciobanu]{Laura Ciobanu$^1$}
\address{
$^1$Technische Universit\"at Berlin, Institut f\"ur Mathematik, 
10623 Berlin, Germany
\and Heriot-Watt University, Department of Mathematics, EH14 4AS Edinburgh, UK}
\email{$^1$ciobanu@math.tu-berlin.de, l.ciobanu@hw.ac.uk}
\author[M. Conte]{Martina Conte$^2$}
\address{
$^2$Fakult\"at f\"ur Mathematik,
Universit\"at Bielefeld,
33501 Bielefeld, Germany}
\email{$^2$mconte@math.uni-bielefeld.de}



\begin{abstract}
    We show that first-order formulae are concise in acylindrically hyperbolic groups and certain extensions thereof. We study further classes of groups, including Burnside groups, icc groups, groups with the `Big Powers' condition, torus knot groups and more, and prove conciseness for wide classes of formulae. 
    
    We also explore properties of definable sets in these groups, such as their finiteness, depending on the type of formula considered.
\end{abstract}

\maketitle

\section{Introduction}
Definable sets in groups, that is, those sets determined by first-order formulae, appear naturally in many algebraic, geometric and model-theoretic contexts, but many of their basic properties, such as when they are finite, are not fully understood in non-commutative infinite groups. In this paper we study conciseness, a concept that relates the finiteness of a definable set to the finiteness of the subgroup it generates. This concept has been widely studied for words, which are a special case of the more general first-order formulae we consider here.  

Let $F$ be the free group on countably many generators, and let $w =w(x_1, \dots, x_n)$ be an element of $F$. We define the \emph{verbal set} $G_w$ to be the set of all $w$\nobreakdash-values in a group~$G$ and we denote by $w(G)$ the \emph{verbal subgroup}   generated by $G_w$:
\[
	G_w = \{w(g_1, \dots, g_n) \mid g_1, \dots, g_n \in G \} \quad\text{and}\quad w(G) = \langle G_w \rangle.
\]
A word $w$ is \textit{concise} in a class of groups $\mathcal{C}$ if, whenever $G$ is a group in $\mathcal{C}$ and $G_w$ is finite, then $w(G)$ is also finite. Concise words were introduced by P. Hall, who conjectured that every word is concise in the class of all groups. The conjecture was disproved by Ivanov \cite{Iva90}, but the question remains open whether every word is concise in the class of residually finite groups.

While many papers have been dedicated to the topic of concise words, the purpose of this work is to explore a generalisation introduced in \cite{CP} which concerns concise first-order formulae in the language of groups (from now on simply formulae); see Section \ref{sec: formulae, preliminaries} for the definitions.
In loc.~cit.~the question is posed whether every formula is concise in the class of residually finite groups and some partial results are provided towards a positive answer to this question. For example, it is shown that every formula is concise in the class of abelian groups and that every existential formula is concise in finitely generated torsion-free nilpotent groups of class 2. 
It is a well-known open question in geometric group theory whether every hyperbolic group is residually finite.
In Section \ref{sec:AH} we prove:

\begin{restatable}{thm}{AHgroups}
\label{thm: all formulae are concise in AH groups}
 All formulae are concise in the class of acylindrically hyperbolic groups.  
\end{restatable}

We then focus on existential formulae, which are the most natural generalisation of words to formulae, and prove that they are always concise in groups embedding ``nicely'' as subgroups of finite index in direct products of torsion-free acylindrically hyperbolic groups with torsion-free abelian groups (see Proposition \ref{prop: existential formulae are concise in groups embedding nicely in AxK}). This result directly implies:

\begin{restatable}{corollary}{existentialdirectproduct}
    Existential formulae are concise in any direct product $A\times\nobreak K$, where $A$ is a torsion-free abelian group and $K$ is a direct product of torsion-free acylindrically hyperbolic groups.
\end{restatable}

Moreover, the same result allows us to prove the following (see Section~\ref{sec: dihedral artin groups})

\begin{restatable}{corollary}{torusknotgroups}
Existential formulae are concise in torus knot groups.
\end{restatable}

One feature of verbal sets of non-trivial words in acylindrically hyperbolic groups is that they are infinite (see Remark \ref{rmk: verbal sets in AH groups are infinite}). Starting from this observation, we investigate the finiteness of more general definable sets. Here the situation differs substantially from the verbal case. If one allows for torsion, instances of finite definable sets can be easily found. However, there are examples of complex formulae that define the trivial set in whole classes of torsion-free groups. In Section \ref{sec: trivial formulae in AH groups} we focus on formulae that define the trivial set in torsion-free acylindrically hyperbolic groups and, more specifically, in free groups, and give some examples of those.      

Finally, in Sections \ref{sec: groups with BP} and \ref{sec: icc groups} we look at conciseness in other related classes of groups, namely groups with the `Big Powers' condition (see Definition \ref{def: alternative BPC}) and icc groups, i.~e. groups in which every non-trivial conjugacy class is infinite.

\subsection*{Notation} 

If not otherwise specified, given a natural number $n$, $F_n$ will denote the free group on $n$ generators and $F$ the free group on countably many generators.
The centre of a group $G$ is denoted by $\mathrm{Z}(G)$.
If $g$ is an element of a group $G$ we denote by $[g]$ its conjugacy class. 
For any positive integer $n$, $\Z/n\Z$ is the cyclic group of order $n$.

\subsection*{Acknowledgments} The authors thank Simon Andr\'e for helpful discussions, and thank Jan Moritz Petschick for his proof that weakly branch groups are icc.

The first author benefited from the support provided by the Deutsche Forschungsgemeinschaft (DFG, German Research Foundation) under Germany's Excellence Strategy -- The Berlin Mathematics Research Center MATH+ (EXC-2046/1, EXC-2046/2, project ID: 390685689).
The second author was funded by the Deutsche Forschungsgemeinschaft (DFG, German Research Foundation) – Project-ID~491392403 – TRR~358. 

\section{Preliminaries and basic results} 

\subsection{Formulae}
\label{sec: formulae, preliminaries}

Here we define formulae in the language of groups. Recall that the language of groups is given by $\mathcal{L}_{\text{gp}}=\lbrace 1, \cdot, ^{-1}\rbrace.$ For our purposes we regard a first-order formula in the language of groups as a finite string of symbols built using the symbols of $\mathcal{L}_{\text{gp}}$, variable symbols, the equality symbol ($=$), the connectives ``and'' ($\wedge$), ``or'' ($\vee$), ``not'' ($\neg$), the existential quantifier ($\exists$), the universal quantifier ($\forall$) and parentheses.

If the universal (resp.~existential) quantifier is the only quantifier occurring in a formula $\varphi$, one says that $\varphi$ is a universal (resp.~existential) formula.

A variable in a formula is said to be \textit{free} if it is not bound to any quantifier. A \textit{sentence} is a formula without free variables. The \textit{(first-order) theory} of a group $G$ is the set of sentences that are satisfied in $G$. If two groups $G_1$ and $G_2$ satisfy the same sentences (resp.~the same universal sentences), one says that $G_1$ and $G_2$ are \textit{elementarily equivalent} (resp.~\textit{universally equivalent}).  

We will simply call \emph{formula} a first-order formula in the language of groups with one free variable and without parameters (i.e., when evaluating formulae in a group $G$, no constants from $G$ are allowed). If $\varphi$ is a formula and $x$ is a free variable, we write $\varphi(x)$. In case we need to specify the variables $\overline{y}$ occurring in $\varphi$ that are bound to a quantifier, we will write $\varphi(x;\overline{y})$. 
Therefore, in our setting a formula $\varphi(x)$ is equivalent to one given by an expression of the following form:
\begin{equation}
\label{eq: general form of a formula}
\varphi(x):=\overline{Q}\ \overline{y}:\bigvee_{j=1}^m \left(w_{j, 1}(x,\overline{y})\#_{j,1} 1\wedge\cdots\wedge w_{j,k_j}(x,\overline{y})\#_{j,k_j}1\right),
\end{equation}
where $\overline{Q}$ is any string of quantifiers in $\lbrace\forall,\exists\rbrace$, every $w_{j,i}$ is a word, $m$ and the $k_j$'s are non-negative integers, and $\#_{j,i}\in \lbrace =, \neq\rbrace$.
We make the convention that the string $\overline{y}$ can be empty and that every element of such string is bound to exactly one quantifier.

A \emph{positive} formula is a formula where no inequalities occur, i.e., using the same notation as before, a formula equivalent to an expression of the form
\begin{equation}
\label{eq: general form of a positive formula}
\varphi(x):=\overline{Q}\ \overline{y}:\bigvee_{j=1}^m \left(w_{j,1}(x,\overline{y})= 1\wedge\cdots\wedge w_{j,k_j}(x,\overline{y})=1\right).
\end{equation}

\noindent The set of positive sentences satisfied by a group $G$ form the \textit{positive (first-order) theory} of $G$.

Given a group $G$, the set defined by a formula $\varphi$ is given by 
$$G_\varphi:=\lbrace g\in G\mid \varphi(g) \ \text{is true} \rbrace.$$ 
Following \cite{CP}, we denote by $\varphi(G)$ the group generated by $G_\varphi$ and we refer to it as the \emph{logical subgroup} of $\varphi$.
\begin{definition}
    A formula $\varphi$ is \emph{concise} in a class of groups $\mathcal{C}$ if, whenever $G$ is a group in $\mathcal{C}$, the finiteness of $G_\varphi$ implies the finiteness of $\varphi(G)$. 
\end{definition}

Note that any word $w\in F_n$ can be regarded as an existential formula $\hat{w}(x):=\exists y_1,\ldots, y_n\colon x=w(y_1,\ldots, y_n)$ which, by slight abuse of notation, we will denote again by $w$, and which we will call the \textit{word formula associated to $w$}. Hence we will identify the verbal set $G_w$ with the set defined by this formula and the corresponding verbal subgroup with the respective logical subgroup.

Following \cite[Section 2.1]{MN}, we say that a word $w$ is \emph{proper} if there exist groups $G$ and $H$ such that $G_w\neq 1$ and $H_w\neq H$.

\subsection{Schur's Lemma and infinite conjugacy classes}
\label{subsect: prelim: schur's lemma and icc}

One of the key tools to show conciseness of words and, more generally, formulae, is to use the existence of (some) infinite conjugacy classes in a group.

We collect some basic but key observations below.

\begin{lemma} Let $G$ be a group and $\varphi$ a formula. 
    \begin{itemize}
    \item[1.] The set $G_\varphi$ defined by $\varphi$ is closed under conjugation. 
    \item[2.] The logical subgroup $\varphi(G)$ is normal.
    \end{itemize}
\end{lemma}

In light of the following result (see \cite{CP}, Lemma 2.6), also known as Dietzmann's Lemma (see \cite[Corollary 2 to Lemma 2.14]{Rob72}), in order to test if a formula $\varphi$ is concise in a group $G$, it is enough to consider the case when $G_\varphi$ contains elements of infinite order. 

\begin{lemma}[Schur reduction]\label{lem:schur_reduction}
	Let $S$ be a finite normal set in $G$. The subgroup $H = \langle S \rangle$ is finite if and only if the image of $s$ in $H/H'$ has finite order for all $s \in S$.
\end{lemma}
This has the following consequence.

\begin{corollary}\label{cor:infconjclass}
    Let $G$ be a group in which every infinite order element has infinite conjugacy class.
    Then every formula is concise in $G$.
\end{corollary}

\begin{proof}
  Let $\varphi$ be a formula and assume that $G_\varphi$ is finite. 
  If $G_\varphi$ contains some element $g$ of infinite order, then the entire conjugacy class of $g$ must be contained in $G_\varphi$ as well. Since the conjugacy class of $g$ is infinite, this leads to a contradiction.
  Hence all elements in $G_\varphi$ have finite order, and $\varphi(G)$ is finite by Lemma \ref{lem:schur_reduction}.  
\end{proof}

Given positive integers $d$ and $n$, the Burnside group $\mathrm{B}(d,n)$ is defined as the quotient of the free group on $d$ generators by the normal subgroup generated by all the $n$th powers. A Burnside group is a group that has this form for some $d$ and $n$. From Lemma \ref{lem:schur_reduction} it immediately follows:

\begin{corollary}
    Every formula is concise in torsion 
    (infinite) groups.
    In particular, every formula is concise in any Burnside group.
\end{corollary}

 A group is called \emph{icc} (which stands for `infinite conjugacy classes') if every (non-trivial) conjugacy class is infinite.

For a group $G$ and an element $g \in G$, denote by $[g]$ the conjugacy class of $g$. In any group $G$ and for any formula $\varphi$, if $g\in G_\varphi$ then its conjugacy class $[g]$ is also in $G_\varphi$, so if $G$ is icc then $G_\varphi$ is either trivial or infinite. This immediately implies

\begin{corollary}\label{icc:concise}
     All formulae are concise in icc groups. 
\end{corollary}

In particular, in icc groups that do not satisfy a law  all verbal sets must  be infinite. Examples of icc groups are given in Section \ref{sec: icc groups}.

\subsection{Conciseness under group extensions}

In this section we collect some results on how conciseness behaves under taking group extensions. More results in this direction are contained in Proposition \ref{prop: infinite order elem in AH gp has infinite conj class} and in Section~\ref{sec: split central ext}. By a group extension $G$ of $H$ by $N$ we mean a short exact sequence 
$$1\rightarrow N\rightarrow G\rightarrow H\rightarrow 1.$$

We start by observing that conciseness of positive formulae behaves well under taking direct products. 

\begin{lemma}
    \label{lem: concise positive formula in direct product is concise}
    Let $\mathcal{C}_1,\ldots, \mathcal{C}_n$ be $n$ classes of groups, $\varphi$ a positive formula that is concise in each of these classes and let $G_i\in\mathcal{C}_i$ for every $i\in\lbrace 1,\ldots n\rbrace$. Then $\varphi$ is concise in $G:=\prod_{i=1}^n G_i$.
\end{lemma}

\begin{proof}
 We show that, for any positive  formula $\varphi$, the equality $\prod_{i=1}^n{(G_i)_\varphi}=(\prod_{i=1}^nG_i)_\varphi$ always holds.
  Let $\varphi$ be any positive formula, equivalent to

   \begin{equation*}
\overline{Q}\ \overline{y}:\bigvee_{j=1}^m \left(w_{j,1}(x,\overline{y})= 1\wedge\cdots\wedge w_{j,k_j}(x,\overline{y})=1\right).
\end{equation*}

Since $G=\prod_{i=1}^n G_i$, any element $g\in G$ can be written uniquely as a product $g_1\cdots g_n\in G_1\cdots G_n$ and every element $h^{(j)}$ in a tuple $\overline{h}$ of elements of $G$ can be written uniquely as $h^{(j)}_1\cdots h^{(j)}_n\in G_1\cdots G_n$. We write for short $\overline{h}=\overline{h}_1\cdots\overline{h}_n$, where $\overline{h}_i$ denotes the tuple $(h^{(j)}_i)_{j=1}^\ell$, and $\ell$ is the length of $\overline{h}$.
If $g\in G$ is a solution of $\varphi$, the formula can be written component-wise as  
\begin{equation*}
\overline{Q}\ \overline{h}_1\cdots\overline{Q}\ \overline{h}_n:\bigvee_{j=1}^m \left(\bigwedge_{i=1}^n\left(w_{j,1}(g_i,\overline{h}_i)= 1\right)\wedge\cdots\wedge \bigwedge_{i=1}^n\left(w_{j,k_j}(g_i,\overline{h}_i)=1\right)\right), 
\end{equation*}
which is in turn equivalent to

\begin{equation*}
\overline{Q}\ \overline{h}_1\cdots\overline{Q}\ \overline{h}_n:\bigvee_{j=1}^m \left(\bigwedge_{i=1}^n\left(\left(w_{j,1}(g_i,\overline{h}_i)= 1\right)\wedge\cdots\wedge \left(w_{j,k_j}(g_i,\overline{h}_i)=1\right)\right)\right). 
\end{equation*}

It follows that $G_\varphi=\prod_{i=1}^n({G_i})_{\varphi}$.
We can conclude that, if $G_\varphi$ is finite, then each $(G_i)_\varphi$ is finite and, since $\varphi$ is concise in every $G_i$, also the logical subgroup $\varphi(G)=\prod_{i=1}^n\varphi(G_i)$ is finite.
  
\end{proof}

\begin{remark}
    Note that the argument used in the proof of the previous Lemma \ref{lem: concise positive formula in direct product is concise}, namely that $(\prod_{i=1}^n{G})_\varphi=\prod_{i=1}^n{(G_i)_\varphi}$, does not extend in general to all formulae. For example, let $p$ be a prime and consider the negative existential formula  
    $$\varphi(x):=\exists y: [x,y]\neq 1\wedge x^p\neq 1.$$
    Let $G_1$ be a torsion-free infinite abelian group and let $G_2$ be an infinite non-abelian $p$-torsion group. 
    An element $g=g_1g_2\in G_1G_2$ in $G_\varphi$ satisfies $[g_2, h_2]\neq 1$ for some $h_2\in G_2$ and $g_1\neq 1$. In particular $G_\varphi$ is infinite. However, $(G_1)_\varphi=(G_2)_\varphi=\emptyset$.
\end{remark}

Since the sets defined by word formulae behave well under taking quotients, we have the following.

\begin{lemma}
\label{lem: conciseness is preserved under finite-by-concise}
 Let $\mathcal{C}$ be a class of groups and $w$ a word that is concise in $\mathcal{C}$. Let $H$ be a finite group, $Q\in\mathcal{C}$ and let $G$ be a group extension $1\rightarrow H\rightarrow G\rightarrow Q\rightarrow 1$. Then $w$ is concise in $G$.
\end{lemma}

\begin{proof}
Suppose that $G_w$ is finite. Then necessarily $Q_w$ is finite and therefore $w(Q)\cong \frac{w(G)H}{H}$ is finite as well because $w$ is concise in $Q$. Since $H$ is finite, this implies that also $w(G)H$, and hence $w(G)$, are finite.    
\end{proof}

From the previous lemma it follows that, for example, every word is concise in any group extension of an icc group by a finite group. In Proposition \ref{prop: infinite order elem in AH gp has infinite conj class} we will extend this result to all formulae.

\begin{lemma}
    \label{lem: inf ord has inf con cl-by-icc is concise}
   Let $\varphi$ be any formula. Let $H$ be a group where every infinite order element has infinite conjugacy class, $Q$ an icc group and let $G$ be an extension $1\rightarrow H\rightarrow G\rightarrow Q\rightarrow 1$. Then $\varphi$ is concise in $G$.
\end{lemma}

\begin{proof}
    Let $g\in G_\varphi$ be of infinite order and consider its class $\overline{g}$ in $Q\cong G/H$. If $g$ does not belong to $H$ then the conjugacy class of $\overline{g}$ is infinite and hence the conjugacy class of $g$ is infinite. Otherwise $g\in H\cap G_\varphi$. Therefore $g$ has infinite conjugacy class and, since $G_\varphi$ is normal, this set must be infinite.    
\end{proof}

We will often use the following straightforward remark.

\begin{remark}
    Let $\varphi$ be an existential formula that takes infinitely many values in a group $H$ and let $G$ be a group containing $H$. Then $G_\varphi$ is infinite.
\end{remark}

\section{Formulae in acylindrically hyperbolic groups}\label{sec:AH}

The class of acylindrically hyperbolic groups is extensive (see \cite[Section 8]{OsinAH}): it includes all non-elementary relatively hyperbolic groups, non-(virtually
cyclic) groups acting properly on proper CAT(0)-spaces with at least one rank 1-element (which include (non-affine) irreducible Coxeter groups and many classes of Artin groups), mapping class groups of compact surfaces of genus at least $1$, outer automorphism groups
of free groups of rank $\geq 2$, many groups acting on simplicial trees (\cite{MO}), graph products that do not split as direct products, etc.

An action denoted by $\circ$ of a group $G$ on a metric space $(\cS,\d)$ is called {\it acylindrical} if for every $\e> 0$ there exist $R\geq 0$ and $N \geq 0$ such that for every two points $x,y\in \cS$ with $\d(x,y)\geq R$ there are at most $N$ elements of $G$ satisfying
$\d(x, g\circ x)\leq \e$ and  $\d(y, g\circ y) \leq \e .$ 
A group $G$ is called {\it acylindrically hyperbolic} (term introduced in \cite{OsinAH}) if it admits a non-elementary acylindrical action on a hyperbolic space (in this situation non-elementary is equivalent to $G$ being non-virtually cyclic and the action having unbounded orbits).

In acylindrically hyperbolic groups conjugacy classes of elements of infinite order are infinite, hence we can apply Schur's result (Lemma \ref{lem:schur_reduction}).
To show that conjugacy classes are infinite we use the following result. Recall from Section \ref{subsect: prelim: schur's lemma and icc} that a group is called icc if every (non-trivial) conjugacy class is infinite.

\begin{theorem}\label{main:AH}
Let $G$ be an acylindrically hyperbolic group and $K(G)$ the \emph{finite radical} of $G$, that is, the maximal finite normal subgroup of $G$. 
\begin{itemize}
    \item[1.]\cite[Theorem 2.35]{DGO}    The group $G$ is icc if and only if $K(G)=\{1\}.$

\item[2.] \cite[Lemma 5.10]{H}  The group $G/K(G)$ is acylindrically hyperbolic and has trivial finite radical. 
\end{itemize}
\end{theorem}

For example, $F \times \Z/2\Z$ is hyperbolic but not icc.

Now it follows from Lemma \ref{lem: conciseness is preserved under finite-by-concise} that every word is concise in the class of acylindrically hyperbolic groups. However even more is true.

\begin{proposition}
\label{prop: infinite order elem in AH gp has infinite conj class}    
   Let $g$ be an infinite order element in an acylindrically hyperbolic group $G$. Then the conjugacy class $[g]$ of $G$ is infinite. 
\end{proposition}

\begin{proof}
Let $\ \bar{}:G \mapsto G/K(G)$ be the projection map.

Let $g \in G$ have infinite order. Then clearly $g \notin K(G)$ and $\bar{g}$ is not trivial in $G/K(G)$. Since by Theorem \ref{main:AH} (2) the group $G/K(G)$ is acylindrically hyperbolic and has trivial finite radical, we get by part (1) of the same theorem that $G/K(G)$ is icc, so in particular $[\bar{g}]$ is infinite.

As conjugacy classes map to conjugacy classes under the projection $\ \bar{ }$, the conjugacy class of $g$ in $G$ must also be infinite.
\end{proof}

By Corollary \ref{cor:infconjclass} we therefore get: 

\AHgroups*

The same argument applies to all groups that are extensions of a finite group by a non-trivial icc group, in particular to groups that have a finite hyper FC-centre. These cannot be virtually nilpotent groups (see~\cite{FF}).

It is clear that all formulae are also concise in any direct product of acylindrically hyperbolic groups.
Moreover, combining Proposition \ref{prop: infinite order elem in AH gp has infinite conj class} with Lemma \ref{lem: inf ord has inf con cl-by-icc is concise} we can conclude that every formula is concise in any extension $1\rightarrow H\rightarrow G\rightarrow Q\rightarrow 1$, where $H$ is an acylindrically hyperbolic group (possibly with torsion) and $Q$ is a torsion-free acylindrically hyperbolic group.

\begin{remark}
    \label{rmk: verbal sets in AH groups are infinite}
    \normalfont
    In non-elementary hyperbolic groups, and even more generally in acylindrically hyperbolic groups, verbal subgroups are always infinite. 
    
    This follows from the fact that the verbal width of any word is infinite (see \cite{MN} for non-elementary hyperbolic groups and \cite{BBF} for acylindrically hyperbolic groups). 
This implies that verbal sets of proper words in acylindrically hyperbolic groups are infinite. Indeed, suppose that $G$ is an acylindrically hyperbolic group and $w$ is a proper word such that $G_w$ is finite. Then, by Proposition \ref{prop: infinite order elem in AH gp has infinite conj class} every element of $G_w$ must have finite order. It follows by Schur's Lemma \ref{lem:schur_reduction} that the verbal subgroup generated by $G_w$ is finite, which would contradict the fact that $w$ has infinite width. 

By using results of Ol'shanski\u i on quasi-geodesics in hyperbolic groups (\cite{O}) and the fact that any acylindrically hyperbolic group $G$ contains a non-elementary hyperbolic group $H$ that is hyperbolically embedded, and therefore quasi-isometrically embedded in $G$ (see for example \cite[Lemma 6.6]{AC2017}, \cite[Lemma 3.1]{AMS} and \cite[Theorem 6.14]{DGO}), one could give a geometric proof of this fact, and also show that every verbal set of a non-trivial word in any acylindrically hyperbolic group contains elements of arbitrarily large word length. Here, by word length we mean the following. Given an element $g$ in a group $G$ with generating set $\mathcal{A}$, the word length of $g$ in $G$ with respect to $\mathcal{A}$ is the distance of $g$ from the identity element $e$ in the Cayley graph~$\mathcal{C}(G,\mathcal{A}).$
\end{remark}

As a corollary of this remark we obtain the following.

\begin{corollary}\label{cor:extAH}
    Every word is concise in any group containing an acylindrically hyperbolic group as a subgroup.
\end{corollary}

\section{Direct products of acylindrically hyperbolic groups with abelian groups}
\label{sec: split central ext}

In this section we consider direct products of the form $A\times H$, where $A$ is torsion-free abelian and $H$ is torsion-free acylindrically hyperbolic. We prove that all words, and, more generally, all existential formulae, are concise in this class.
A direct product as above is neither icc nor acylindrically hyperbolic as it has an infinite center isomorphic to $A$. In Section \ref{sec: dihedral artin groups} we will use the results of this section to prove that existential formulae are concise in torus knot groups. 

First, we observe that by Lemma \ref{lem: concise positive formula in direct product is concise} every positive formula $\varphi$ is concise in any direct product $A\times K$, where $A\leq \mathrm{Z}(G)$ and $K$ is a group where $\varphi$ is concise.
Since every word $w$ can be represented by a positive existential formula, we immediately obtain:

\begin{proposition}
    Every word is concise in any direct product $A\times K$, where $A$ is abelian and $K$ is a group where every word is concise.
\end{proposition}

We show that all existential formulae are concise in such direct products following a similar approach to the one used in \cite{CP} to prove the conciseness of all formulae in abelian groups and of existential formulae in finitely generated torsion-free nilpotent groups of class $2$. Namely, we show that, if an element of infinite order satisfies a given existential formula, then infinitely many of its powers must satisfy the same formula.
As in the nilpotent case, we will make use of the fact that we can restrict to solutions of the formula lying in the center of the group (see \cite[proof of Proposition 3.10]{CP}).

\begin{proposition}
\label{prop: existential formulae are concise in groups embedding nicely in AxK}
 Let $G$ be a group such that every $g\in G\setminus \mathrm{Z}(G)$ of infinite order has infinite conjugacy class. Let $A$ be a torsion-free abelian group and assume that $G$ embeds into a direct product $A\times K$ as a subgroup of finite index. Moreover, suppose that $\mathrm{Z}(G)$ embeds into $A$.

Let $\varphi$ be an existential formula with free variable $x$, so $\varphi$ is equivalent to $$\exists\overline{y}:\bigvee_{j=1}^m \left(w_{j,1}(x,\overline{y})\#_{j,1} 1\wedge\cdots\wedge w_{j,k_j}(x,\overline{y})\#_{j,k_j}1\right),$$
     where every $w_{j,i}$ is a word, $m$ is a positive integer, and $\#_{j,i}\in \lbrace =, \neq\rbrace$. 
     
     Then $\varphi$ is concise in $G$.
    
\end{proposition}

\begin{proof}
    Let $g\in G_\varphi$ be of infinite order. If $g\notin \mathrm{Z}(G)$, then its conjugacy class is infinite and there is nothing to prove. So suppose that $g\in \mathrm{Z}(G)$. 
    Since $g$ satisfies $\varphi$, we can assume that it satisfies one of the conjunctive terms occurring in the formula, i.e., $$w_{j,1}(g,\overline{h})\#_{j,1} 1\wedge\cdots\wedge w_{j,k_j}(g,\overline{h})\#_{j,k_j}1$$ for some index $j$ and a certain tuple $\overline{h}$ of elements of $G$. To simplify notation we will omit the index $j$. 
    Since $g$ is central, we can rewrite the previous conjunction as 
    $$g^{m_1}\#_{1}\tilde{w}_1(\overline{h})\wedge\cdots\wedge g^{m_k}\#_{k}\tilde{w}_k(\overline{h})$$ 
    for some integers $m_i$ (namely, $-m_i$ is the sum of the exponents of $g$ in $w_{i}$, for $i\in\lbrace 1,\ldots, k\rbrace$) and some words $\tilde{w}_i$.
    By abuse of notation we write the words $\tilde{w}_i$ again as $w_i$. 
    If all the exponents $m_i$ are equal to $0$, the given formula becomes a sentence true in $G$ and therefore every element of the centre is in $G_\varphi$, which is therefore infinite. 

    Assume then that at least one exponent $m_i$ is different from $0$.
    Since $G$ embeds into $A\times K$, each $h_l$ (the $l$th element in the tuple $\overline{h}$) can be written uniquely as a product $z_lv_l\in AK$. Therefore, for each $i$ we can write $$w_i(\overline{h})=w_i(\overline{zv})=w_i(\overline{z})w_i(\overline{v})\in AK,$$ where $\overline{zv}$ denotes the tuple $(z_lv_l)_{l}$. 
    
    Rearrange the terms of the conjunction in a way that, in the first $p$ terms, the symbol $\#_{i}$ is an equality and, in the remaining $k-p$ terms, it is an inequality $\neq$, where $p\in\lbrace 0,\ldots, k\rbrace$.
    Then $g^{m_i}=w_i(\overline{h})$ for each $i\in\lbrace 1,\ldots, p\rbrace$.
    It follows from the unique representation of $g^{m_i}$ that, for each of these indices $i$, $w_i(\overline{v})=1$ and $g^{m_i}=w_i(\overline{z})$. For any power $g^N$ of $g$ we have: $(g^N)^{m_i}=w_i(\overline{z})^N=w_i({\overline{z}}^N)=w_i({\overline{z}^N})w_i({\overline{v}})=~w_i({\overline{z}}^N{\overline{v}})$. 

    Now consider the indices $i\in\lbrace p+1,\ldots, k\rbrace$. For each of these $i$ we have $g^{m_i}\neq w_i(\overline{h})$. Write as before $\overline{h}=\overline{z} \overline{v}$ 
    and suppose that for some integer $N$ we have $(g^N)^{m_i}=w_i({\overline{z}}^N)w_i({\overline{v}})=w_i({\overline{z}}^N{\overline{v}})$. This implies that $w_i({\overline{v}})=~1$. 
    Therefore the initial inequality becomes $g^{m_i}\neq w_i(\overline{z})$. However, from $(g^N)^{m_i}=(g^{m_i})^N=w_i(\overline{z})^N$ follows that $g^{m_i}=w_i(\overline{z})$, a contradiction. Hence we can conclude that $(g^N)^{m_i}\neq w_i({\overline{z}}^N\overline{v})$ for every~$N$.

    We know that $\overline{h}=\overline{z}\overline{v}$ is a tuple of elements in $G$. Now, for every $N\geq~M:=[(A\times K):G]$, every $\overline{z}^N$ in $A$ is a tuple of elements of $G$ and so $\overline{z}^{N+1}\overline{v}=\overline{z}^N(\overline{z}\overline{v})$ is also a tuple of elements of $G$.    

    In conclusion, for every $N$ big enough (namely, $N\geq M+1$) the element $g^N$ satisfies the following formula in $G$:

    $$\bigwedge_{i=1}^p{\left((g^N)^{m_i}= w_i({\overline{z}}^N\overline{v})\right)} \wedge \bigwedge_{i=p+1}^k {\left((g^N)^{m_i}\neq w_i({\overline{z}}^N\overline{v})\right)}$$
and therefore satisfies $\varphi$. It follows that, if $G_\varphi$ contains an element of infinite order, then it is infinite, and therefore $\varphi$ is concise.
    \end{proof}

 \begin{corollary}
    \label{cor: existential is concise in split central extension}
    Let $G$ be equal to a direct product $\mathrm{Z}(G)\times K$, where $\mathrm{Z}(G)$ is torsion-free and every $g\in G\setminus \mathrm{Z}(G)$ of infinite order has infinite conjugacy class.
    Then every existential formula is concise in $G$.
\end{corollary}

\begin{remark}
    Recall that the set of elements of a group $G$ that have finite conjugacy class forms a normal subgroup denoted by $\mathrm{FC}(G)$ and called the FC-\nobreak centre of $G$. Note that in the groups considered in Corollary \ref{cor: existential is concise in split central extension}, the fact that every $g\in G\setminus \mathrm{Z}(G)$ of infinite order has infinite conjugacy class is equivalent to the condition $\mathrm{FC}(G)=\mathrm{Z}(G)$. Indeed, we already know that infinite order elements have infinite conjugacy class. Finite order elements in $G$ that are not in the centre have the form $(1,k)$, where $k$ is any non-trivial element of finite order in $K$. But for any element $k$ of $K$, and for any non-trivial $a\in A$, the conjugacy class of the infinite order element $(a,k)$ must be infinite, so the conjugacy class of $k$ is also infinite.
\end{remark}

    \existentialdirectproduct*    

    \begin{proof}
    The centre of $G$ is isomorphic to $A$ and $K$ is an icc group. Hence the result follows from the previous Corollary \ref{cor: existential is concise in split central extension}.
\end{proof}

\section{Trivial formulae in acylindrically hyperbolic groups}
\label{sec: trivial formulae in AH groups}

In Remark \ref{rmk: verbal sets in AH groups are infinite} we observed that verbal sets of acylindrically hyperbolic groups are infinite since, in this class, proper words have infinite width.
However, we cannot conclude that the set $G_{\varphi}$ is infinite for any formula $\varphi$, as the arguments involving width do not directly apply to definable sets.
Indeed, it is possible to find non-trivial finite definable sets. For instance, for any positive integers $n,m\geq 1$ and free group $F_n$, in the hyperbolic group $F_n\times \Z/m\Z$ the center and the set of elements of order $m$, which are both definable sets, are finite and non-trivial. 

In this section we consider instances of formulae for which we get trivial definable sets, and say that a formula $\varphi$ is \emph{trivial} in  a class of groups $\mathcal{C}$ if $G_\varphi=\lbrace 1\rbrace$ for every $G\in\mathcal{C}$.\footnote{The notion of trivial formulae introduced here should not be confused with the concept of trivial positive first-order theory, where a group is said to have trivial positive first-order theory if it has the same positive theory as finitely generated non-abelian free groups.}

\begin{example}
    Consider the formula $\varphi$ given by $\forall x: xy=yx$ which defines the center $\mathrm{Z}(G)$ in any group $G$. If $G$ is icc then $\mathrm{Z}(G)=\{1\}$, hence $\varphi$ is trivial in the class of icc groups.
\end{example}

\begin{remark}
Let $\mathcal{F}$ be the class of finitely generated non-abelian free groups.
By the work of Sela and Kharlampovich-Myasnikov, all finitely generated non-abelian free groups are elementarily equivalent, so a formula $\varphi$ is non-trivial in some $F_n$, $n \geq 2$, if and only if it is non-trivial in $\mathcal{F}$, as every group in this class satisfies the sentence $\exists x: x\neq 1\wedge \varphi(x)$. Similarly, a formula defines the empty set in some $F_n$ if and only if it defines the empty set in all groups belonging to $\mathcal{F}$.
\end{remark}

It is therefore interesting to understand when a definable set of an acylindrically hyperbolic group can be finite or trivial. Since acylindrically hyperbolic groups contain free subgroups (see for example \cite[Theorem 1.2]{OsinAH} and \cite[Theorem 6.14]{DGO}), one possible approach is to compare when a given formula defines a finite or empty set in an acylindrically hyperbolic group versus when it does so in groups belonging to $\mathcal{F}$. 

First, note that a set defined by a positive formula is non-empty in an acylindrically hyperbolic group if and only if it is non-empty in the class $\mathcal{F}$.
Indeed, consider the positive sentence $\exists x:\varphi(x)$ and use that 
acylindrically hyperbolic groups have trivial positive theory (\cite[Corollary 1.8]{AF}).

Second, for existential formulae we have the following.

\begin{lemma}
\label{lem: existential finite in AH is trivial in free}
    Let $\varphi$ be an existential (or quantifier-free) formula. If there is an acylindrically hyperbolic group $G$ such that the set $G_{\varphi}$ is finite, then $\varphi$ is trivial or defines the empty set in the class $\mathcal{F}$ of finitely generated non-abelian free groups. 
    
\end{lemma}

\begin{proof}

    Suppose by contradiction that $\varphi$ does not define the empty set and is non-trivial in $\mathcal{F}$. By the icc property it follows that the set defined by $\varphi$ in any finitely generated non-abelian free group is infinite. Since any acylindrically hyperbolic group $G$ contains a finitely generated non-abelian free subgroup $F$ and $F_\varphi \subseteq G_\varphi$, the set $G_\varphi$ must also be infinite.
    \end{proof}

The following example shows that, if one allows torsion, the converse to Lemma \ref{lem: existential finite in AH is trivial in free} is not true. 

\begin{example}
\label{ex: in infinite dihedral group there are non-trivial def sets that are trivial in free groups}
Let $D_\infty:=\langle r,s\mid s^2=1,\ r^s=r^{-1} \rangle$ be the infinite dihedral group, and take $G$ to be the free product of $D_\infty$ with some free group $F_n$. This is a relatively hyperbolic group, so acylindrically hyperbolic as well.
\begin{itemize}
  \item[(i)] The set of elements of order $2$ is infinite in $D_\infty$, but clearly trivial in any non-abelian free group. Note that this set is given by a quantifier-free formula. 

\item[(ii)] For existential formulae, consider the formula $\gamma$ given by 
$$\exists y_1,y_2,y_3\colon x=[y_1,y_2]\wedge x=y_3^n,$$ where $n >1$ is some fixed positive integer. It was proved by Schützenberger (\cite{Sch}, see also \cite[Lemma 36.4]{Baumslag60}) that the only solution $(y_1, y_2, y_3)$ to the equation $[y_1,y_2]=y_3^n$ ($n\geq 2$) in free groups is the trivial one, hence $\gamma$ is a trivial formula in $\mathcal{F}$. However, in $D_\infty$ the element $r^{2}$ is a proper power and $r^2=(sr^{-1}s^{-1})r=[s^{-1},r]$.
  
\end{itemize}
The two items above show that there exist acylindrically hyperbolic groups where the converse of Lemma \ref{lem: existential finite in AH is trivial in free} does not hold.
\end{example}

From now on we will focus on torsion-free groups and existential formulae: given an existential formula $\varphi$ and a torsion-free acylindrically hyperbolic group $G$, we ask how the triviality of $\varphi$ in $\mathcal{F}$ affects the triviality of $\varphi$ in $G$.

The converse of Lemma \ref{lem: existential finite in AH is trivial in free} holds true if $\mathcal{U}$ is a subclass of torsion-free acylindrically hyperbolic groups whose members are universally equivalent to finitely generated non-abelian free groups. These are exactly the non-abelian limit groups (\cite[Theorem 5.1]{CG}).

\begin{corollary}
    An existential (or quantifier-free) formula $\varphi$ is trivial in the class of non-abelian limit groups if and only if it is trivial in the class $\mathcal{F}$ of finitely generated non-abelian free groups.
\end{corollary}

\begin{proof}
    Write $\varphi(x;\overline{y})$ as $\exists \overline{y}: \phi(x;\overline{y})$, where $\phi$ is quantifier-free and we allow for the case where $\overline{y}$ does not occur in $\phi$.
    
    Suppose that $\varphi$ is trivial in the class of non-abelian limit groups.
    Then it follows from the previous Lemma \ref{lem: existential finite in AH is trivial in free} that $\varphi$ defines the trivial or empty set in groups in $\mathcal{F}$. If $\varphi$ defines the empty set in a group $F\in\mathcal{F}$, then the universal sentence
    $\forall x, \overline{y}: \neg\phi(x;\overline{y})$
    holds true in $\mathcal{F}$, and hence in non-abelian limit groups, a contradiction. 
    
    For the other direction, suppose that the formula $\exists\overline{y}\colon\phi(x;\overline{y})$ is trivial in~$\mathcal{F}$. Hence the sentence $$\exists x,\overline{y}\colon\phi(x;\overline{y})\implies x=1$$ 
    is true in the class $\mathcal{F}$ of finitely generated non-abelian free groups. This sentence is equivalent to the universal sentence 
    $$\forall x,\overline{y}\colon\neg\phi(x;\overline{y})\vee x=1,$$ 
    which therefore must hold also in the class of limit groups.
\end{proof}

This observation yields the following natural question.

\begin{question}
    Which existential formulae are trivial in free groups?
\end{question}

Even if a general answer to this question seems out of reach, we focus on special kinds of existential formulae given by positive boolean combinations of word formulae and provide some examples of existential formulae of this form that are (non-)trivial in $\mathcal{F}$. These formulae can be written as
\begin{equation}
\label{eq: positive boolean combinations of words}
\exists\overline{y}:\bigvee_{j=1}^m \left(x=w_{j,1}(\overline{y})\wedge\cdots\wedge x=w_{j,k_j}(\overline{y})\right), 
\end{equation}
where $m$ is a positive integer and the $w_{j,i}$ are words on disjoint sets of variables.
It is clear that such a formula is trivial if and only if each term $$\mathrm{b}(x):=\exists\overline{y}:~\left(x=w_{j,1}(\overline{y})\wedge\cdots\wedge x=w_{j,k_j}(\overline{y})\right)$$ 
in the disjunction is trivial. We can therefore restrict our attention to these terms, and we say that such a formula is the (conjunctive) formula \textit{associated to the words $w_{j,1},\ldots, w_{j, k_j}$}, or, equivalently, the formula associated to the equation $w_{j,1}(\overline{y})=\cdots =~w_{j,k_j}(\overline{y})$. Note that $\mathrm{b}$ is the conjunction of the word formulae associated to $w_{j,1},\ldots, w_{j,k_j}$.

First, we observe that, in order for $\mathrm{b}$ to be trivial, it must be associated to at least two (non-trivial) words. Otherwise $\mathrm{b}$ would define a verbal set, which is infinite in any non-abelian free group.

We also remark that we can exclude the case when all words associated to $\mathrm{b}$ are non-commutator words. A \textit{non-commutator word} is a word that is not contained in the commutator subgroup of the free group (or, equivalently, a word in which the sum of the exponents of at least one variable is not zero). P.~Hall himself proved that non-commutator words are concise (see \cite{Rob72} Lemma 4.27). More generally, in \cite{CP} it is shown that positive boolean combinations of non-commutator word formulae are concise (Proposition 2.8 and Example 2.9). The fact that the conjunction of non-commutator word formulae is concise relies on the fact that the set defined by such formulae in non-torsion groups must be infinite. This implies that these formulae never define a trivial set in $\mathcal{F}$.

Therefore the easiest case to consider is a formula defining elements that can be expressed simultaneously as a product of commutators and a product of proper powers.

\begin{example}
   Consider the formula $\mathrm{b}(x)$ given by $$\exists y_1,y_2,y_3\colon x=[y_1,y_2]\wedge x=y_3^n,$$ where $n$ is any fixed positive integer greater than $1$. As already recalled in Example \ref{ex: in infinite dihedral group there are non-trivial def sets that are trivial in free groups}, the only solution $(y_1, y_2, y_3)$ to the equation $[y_1,y_2]=y_3^n$ ($n\geq 2$) in free groups is the trivial one, hence $\mathrm{b}$ is a trivial formula in $\mathcal{F}$.
\end{example}

\begin{example}
\label{ex: commutator=product of two powers}
The next case to consider is the equation $[y_1,y_2]=y_3^ny_4^n$ for $n\geq 2$.
For $n=2$, generalising a result of Lyndon and Newman (\cite{LN}), it was proven in \cite{Sar} that, if $F_2$ is the free group with generators $x,y$, the equation $[x^r,y^s]=a^2b^2$ has a solution if and only if $rs$ is even. In particular, the formula associated to the commutator word and the product of two squares is never trivial in non-abelian free groups. 
More generally, according to \cite[Corollary 11]{Bar}, the equation $$[x^r,y^s]=a^nb^n$$ has a solution in $F_2=\langle x,y\rangle$ if and only if $n\vert r$ or $n\vert s$. In particular, the formula associated to the words $[y_1,y_2]$ and $y_3^ny_4^n$ is never trivial in non-abelian free groups.     
\end{example}

Going further, if we consider the product of three powers, it is well known that, in any group, every commutator can be written as a product of three squares. If instead one looks at products of three cubes, by \cite{AM} the equation $[y_1,y_2]=y_3^3y_4^3y_5^3$ has non-trivial solutions in $\mathcal{F}$. Therefore the formulae associated to these equations are never trivial in $\mathcal{F}$. 

If, on the other hand, one considers equations of the form $\gamma=z^n$, where $\gamma$ is a product of commutators, the problem of determining the minimal number of commutators in the product on the left hand side for which this equation can have a non-trivial solution in free groups was completely solved in \cite{DH}. Namely, it is proven in loc.~cit.~that if in a free group there exist elements $x_1,\ldots, x_k,y_1,\ldots, y_k, z$ such that 
\begin{equation}
\label{eq: prod of comm = power}
[x_1,y_1]\cdots[x_k,y_k]=z^n
\end{equation}
and $n\geq 2k$, then $z=1.$ In our terminology, this means that the formula associated to the equation (\ref{eq: prod of comm = power}) is trivial in $\mathcal{F}$ whenever $n\geq 2k$.

Putting these results together we can construct further instances of trivial and non-trivial existential formulae in $\mathcal{F}$.

\begin{proposition}
    Consider the formula $\mathrm{b}$ associated to the equation
    $$\prod_{i=1}^n{[x_i,y_i]}=\prod_{j=1}^mz_j^{a_j}\prod_{k=1}^l{[u_k,v_k]},$$ where $n,m,l$ are non-negative integers, the exponents $a_j$ are integers and all the variables are pairwise distinct. 
    
    \begin{enumerate}
    
    \item If both $n$ and $l$ are not zero, then $\mathrm{b}$ is not trivial in $\mathcal{F}$.
    
    \item If $l=0$, $n\geq 1$, and $m\geq 2$, then $\mathrm{b}$ is not trivial in $\mathcal{F}$. 
    
    \item If $l=0$, $n\geq 1$, $m=1$ and $2n\leq a_1$, then $\mathrm{b}$ is trivial in $\mathcal{F}$.

    \end{enumerate}
\end{proposition}

\begin{proof}
Let $F$ be any group in $\mathcal{F}$. 
\begin{enumerate}

    \item Assume without loss of generality that $1\leq n\leq l$.  
    Then we can choose two non-commuting elements $x_1, y_1$ in $F$ and set $u_1=x_1$, $v_1=y_1$, and every other variable equal to $1$. The case $l\leq n$ is identical.

    \item Set $z_j=1$ for all $j\geq 3$,  $x_i=y_i=1$ for $i\geq 2$, and choose $z_1$ and $z_2$ to be two non-commuting elements in $F$. Because of the `Big Powers' condition (see Definition \ref{def: alternative BPC}), for any big enough $c$, the product $(z_1^{a_1})^c(z_2^{a_2})^c$ is not trivial. Then one can find a non-trivial solution to the equation $[x_1,y_1]=(z_1^{a_1})^c(z_2^{a_2})^c$ taking $x_1$ and $y_1$ to be suitable powers of two distinct generators of $F$, and using the result recalled in Example \ref{ex: commutator=product of two powers}. More precisely, if $x$ and $y$ are two elements of $F$ belonging to a free basis, we can take $x_1=x^r$ and $y_1=y^s$ with $c$ dividing at least one between $r$ and $s$.

    \item Our equation reduces to equation (\ref{eq: prod of comm = power}) and we can directly use the fact that this has only the trivial solution for $a_1\geq 2n$. 
\end{enumerate}
   
\end{proof}

\section{Further classes of groups}

In this section we explore conciseness of formulae in classes of groups that are related to the classes seen before, namely groups with the `Big Powers' condition, icc groups and torus knot groups.

\subsection{Groups with the `Big Powers' condition}
\label{sec: groups with BP}
In this section we prove that every verbal set in groups satisfying the `Big Powers' condition is necessarily infinite. 

In free groups the following, so-called `Big Powers', condition holds true.

\begin{proposition}[\cite{B}]
 Let $g_1, \ldots, g_k$ be elements in a free group $F_n$, and let $f$ be an element that does not commute with any of them. Then there exists $N\in\mathbb{N}$ such that, for any $n_i\in\Z$ with $|n_i|\geq N, m \in \mathbb{N}$ and $j_i \in \{1, \dots k\}$,
 $$g_{j_1} f^{n_1}g_{j_2} f^{n_2} \dots g_{j_m}f^{n_m} \neq 1_{F_n}.$$
\end{proposition}

In \cite{KMS} a more general `Big Powers' condition is introduced.

\begin{definition}[BP Condition]
    \label{def: alternative BPC}
 Let $G$ be a group, $k$ be any positive integer and let $u=(u_1,\ldots, u_k)$ be a sequence of non-trivial elements of $G$. One says that $u$ is 
 \begin{enumerate}
     \item \emph{generic}, if any two consecutive elements in $u$ do not commute, i.e. $$[u_i,u_{i+1}]\neq 1$$ for every $i\in\lbrace 1,\ldots, k-1\rbrace$;

     \item \emph{independent}, if there exists $N=N(u)\in\N$ such that, for any $\alpha_1,\ldots,\alpha_k$ with each $\alpha_i\geq N$, $$u_1^{\alpha_1}\cdots u_k^{\alpha_k}\neq 1.$$
 \end{enumerate}
The group $G$ satisfies the \emph{`Big Powers' condition} (BP condition) if every generic sequence is independent. In this case one says that $G$ is a BP-group.   
\end{definition}

Note that every group satisfying the BP condition must be torsion-free, as every non-trivial element forms a generic sequence.
Examples of groups for which the BP condition holds are free and, more generally, torsion-free hyperbolic groups and all of their subgroups and fully residually free groups (\cite{KMS}). These examples are all acylindrically hyperbolic or infinite cyclic. However, there are examples of groups with this property that are not acylindrically hyperbolic. For instance, any direct product $G\times A$, where $G$ is torsion-free hyperbolic and $A$ is a non-trivial torsion-free abelian group satisfies the `Big Powers' condition (see \cite[Proposition 3]{KMS}) but is not acylindrically hyperbolic. 

Verbal sets in torsion-free abelian groups are either trivial or infinite. Moreover, virtually abelian groups satisfy the BP condition if and only if they are abelian, because these groups satisfy a law (\cite[Theorem 1]{KMS}). For verbal sets in non-abelian BP-groups we have the following.

\begin{proposition}
    Let $G$ be a non-abelian group satisfying the BP condition. Then every verbal set $G_w$ of a non-trivial word $w$ is infinite. In particular, every word is concise in~$G$.
\end{proposition}

\begin{proof}
    
    Let $w(x_1,\ldots,x_n)\in F_n$ be any non-trivial word. We show that $G_w$ is infinite. Assume that the reduced form of $w$ is $\prod_{j=1}^s{x_{i_j}^{k_j}}$, where $s$ is a positive integer, each $i_j$ belongs to $\lbrace 1,\ldots, n\rbrace$ and each $k_{j}$ is an integer. 
    
    Since $G$ is not virtually abelian, it contains infinite sets of pairwise non-commuting elements (\cite{N}), so one can find pairwise non-commuting elements $g_{1},\ldots, g_{n}$ and form a sequence of generic elements $(g_{i_1},\ldots, g_{i_s}, g_{i_{s-1}}, .\ .$ $\ldots, g_{i_1})$.  
     Then there exists a constant $N$ such that, for every $m_1,\ldots, m_{n}$ and every $m'_1,\ldots, m'_n$ satisfying $m_{i_j}k_j>N$ and $-m'_{i_j}k_j>N$ for $j\in\lbrace 1\ldots, s-1\rbrace$ and $(m_{i_s}-~m'_{i_s})k_s>N$, one has 
    $$g_{i_1}^{m_{i_1}k_1}\cdots g_{i_s}^{m_{i_s}k_s} g_{i_s}^{-m'_{i_s}k_s}\cdots g_{i_1}^{-m'_{i_1}k_1}\neq 1,$$
    i.e., $$w(g_1^{m_1},\ldots, g_n^{m_n})\neq w(g_1^{m'_1},\ldots, g_n^{m'_n}).$$
    It follows that one can find infinitely many pairwise distinct word values and hence $G_w$ is infinite.
\end{proof}

\subsection{Icc groups}
\label{sec: icc groups}

Most of the interest in icc groups comes from the theory of von Neumann algebras, as icc discrete groups are exactly those groups whose von Neumann algebra is a factor of type $II_1$.
We have already established that all formulae are concise in icc groups (see Corollary \ref{icc:concise}). 

Examples of icc groups include any free product of non-trivial groups except the infinite dihedral group (\cite[Example 5.15]{P}) (and more generally several groups acting on trees, \cite{C}), and various group extensions, see \cite{P}.

In this section we present various instances of icc groups that are not acylindrically hyperbolic. These provide further examples of groups where all formulae are concise that are not covered by Theorem \ref{thm: all formulae are concise in AH groups}. 
A simple example is given by the direct product of two icc groups. Indeed, the class of icc groups is closed under extensions (\cite[Corollary 1.3]{P}), but a direct product $G_1\times G_2$ can be acylindrically hyperbolic only if one of the factors $G_i$ is finite (\cite[Corollary 7.3 (b)]{O}).
Another family of examples is given by the Baumslag-Solitar groups $\mathrm{BS}(m,n)=\langle a, t\mid t^{-1}a^mt=a^n\rangle$ which, for every $m\neq \pm n$ with $m,n\neq 0$, are icc (\cite[Example 5.13]{P}) but not acylindrically hyperbolic (\cite[Example 7.4]{O}).

 \begin{example}[Infinite non-abelian simple groups and their icc extensions]
Infinite non-abelian simple groups are icc (\cite[Proposition 0.1]{P}) but they are not acylindrically hyperbolic. This follows from the fact that any acylindrically hyperbolic group $G$ is SQ-universal (\cite[Theorem 2.33]{DGO}) (i.e., every countable group can be embedded into a quotient of $G$), which implies that the group cannot be simple.

Now, let $G$ be an extension $1\rightarrow K\rightarrow G\rightarrow Q\rightarrow 1$, where $K$ is an infinite (non-abelian) simple group and $Q$ is any group. Let $\mathrm{FC}(Q)$ be the FC-centre of $Q$, let $\Theta\colon Q\rightarrow \mathrm{Out}(K)$ be the coupling associated to this extension and $\Phi$ be the restriction of $\Theta$ to $\mathrm{FC}(Q)$. Then, according to \cite[Examples 2.5]{P}, $G$ is icc whenever $\Phi$ is injective. However, since acylindrical hyperbolicity is closed under taking infinite normal subgroups (see for example \cite[Corollary 1.5]{O}), all these extensions give examples of icc groups that are not acylindrically hyperbolic. 

More generally, for the same reason, any icc extension of the form (infinite non-acylindrically hyperbolic group)-by-(any group) is an icc group that is not acylindrically hyperbolic. An example of this kind is given by the Euclidean group $\R^n\rtimes O(n)$ (\cite[Example 3.3]{P}).
\end{example}

Clearly a group is icc if and only if its FC-centre is trivial.
It is not difficult to prove that finitely generated just infinite groups that are not virtually abelian have trivial FC-centre (see for example \cite[Lemma 2.13]{SW}). It follows that: 

\begin{corollary}
  Every formula is concise in the class of finitely generated just infinite groups that are not virtually abelian.  
\end{corollary}

Examples of such groups are the Grigorchuk group, various branch groups and certain arithmetic groups modulo their centre, such as $\mathrm{SL}_n(R)$ for $n\geq 3$ and $\mathrm{Sp}_{2n}(R)$ for $n\geq 2$, where $R$ is the ring of integers of an algebraic number field (\cite{W}). 
Note that the Grigorchuk group is a torsion group, which implies that conciseness is trivially satisfied.
More generally, it was communicated to us by J. Moritz Petschick that weakly branch groups are icc. \label{prop:JMweaklybranch}

\subsection{Artin groups and torus knot groups}
\label{sec: dihedral artin groups}

In this section we consider Artin groups and focus on the dihedral ones, i.e.\ the ones that are generated by two elements. 

\begin{definition}
    Let $\Gamma$ be a finite simple graph, with vertex set $V(\Gamma)$ and with edges labelled by integers $m_{i,j} \in \Z_{\geq 2}$. The Artin group $A(\Gamma)$ is the group defined by the following presentation:
    \[ A(\Gamma) = \langle V(\Gamma) \; | \; _{m_{i,j}}(a_i,a_j) =  {}_{m_{i,j}}(a_j,a_i) \; \text{if the edge} \; \{a_i,a_j\} \; \text{is labelled} \; m_{i,j} \rangle,
    \]
    where $_{m_{i,j}}(a,b)$ is the word $abab \dots$ of length $m_{i,j}$. 
    \begin{itemize}
        \item[(i)] If $|V(\Gamma)| = 2$, we say $A(\Gamma)$ is a \emph{dihedral Artin group}.
        \item[(ii)]  An Artin group $A(\Gamma)$ is 
      of \emph{large type} if all edge labels $m_{i,j}$ are $\geq 3$.
      \item[(iii)] An Artin group $A(\Gamma)$ is \emph{spherical} if its corresponding Coxeter group (obtained by making all the generators have order two) is finite.
    \end{itemize}
\end{definition}

 For dihedral Artin groups, the presentation with respect to the standard generating set, for $m\geq 2$, is
\begin{equation}\label{eq:all groups}
    G(m) = \langle a,b \; | \; _{m}(a,b) =  {}_{m}(b,a) \rangle.
\end{equation}

In particular, for $m=2$ we get $\Z^2$ and for $m=3$ the braid group $B_3$ on three strands.

The dihedral Artin groups can be studied from various perspectives: they are one-relator groups with center, (finitely generated free)-by-cyclic groups with center, virtually direct products of hyperbolic groups, central extensions of virtually free groups and also subgroups of finite index in direct products of hyperbolic groups \cite{CHR, CGL}. To see some of these characterisations, use the following alternative presentations for $G(m)$.  For $m$ odd,
    \[ G(m) \cong \langle x,y \; | \; x^{2} = y^{m} \rangle,
    \]
    by setting $x = {}_{m}(a,b), y = ab$. For $m$ even,
    \[ G(m) \cong \langle x,y \; | \; y^{-1}x^{p}y = x^{p} \rangle = \mathrm{BS}(p,p),
    \]
    where $p= \frac{m}{2}$, by setting $x = ab, y = a$; $\mathrm{BS}(p,p)$ denotes a Baumslag-Solitar group. 

Thus we get the exact sequences:
\begin{equation}
\label{eq: exact sequence for G(m) is central extension of free product}
1 \rightarrow \Z=\langle x^2 \rangle \rightarrow G(m) \rightarrow \Z/2\Z \ast \Z/m \Z\rightarrow 1,
\end{equation}
for $m$ odd, and
\begin{equation}
\label{eveneq: exact sequence for G(m) is central extension of free product}
1 \rightarrow \Z=\langle y^p \rangle \rightarrow G(m) \rightarrow \Z/p\Z \ast \Z \rightarrow 1
\end{equation}
for $m$ even,
which shows that such a group is a central extension of a hyperbolic group.

\begin{remark}(on the `Big Powers' condition)
    Dihedral Artin groups do not satisfy the BP condition because of relations of the form $x^2=y^m$ (when $m$ is odd): one can take $(x, y^{-1})$ to be a generic sequence that is not independent. For $m$ even one can take the generic sequence $(x^{-1}, y^{-1}, x, y)$ that is not independent because $x^{kp}y^m=y^mx^{kp}$ holds for any $k,m>0$.
    Each Artin group that is not free has parabolic subgroups (coming from the defining edges) that are dihedral Artin, so Artin groups are not BP-groups (unless they are free).
 However, every dihedral Artin group contains a non-abelian free group, and hence verbal sets of non-trivial words in dihedral Artin groups are infinite. 
\end{remark}

 It is conjectured that all Artin groups are acylindrically hyperbolic, unless they are direct products or spherical. Many large-type Artin groups were proved to be acylindrically hyperbolic, but there  are outstanding  open cases.

All spherical Artin groups have an infinite center so they are neither acylindrically hyperbolic nor icc.
However, since they are linear (\cite{CW}), by \cite[Theorem 1.4.2]{S} we get
\begin{corollary}
    Every word is concise in every spherical Artin group.
\end{corollary}

We now look at formulae in dihedral Artin groups. We start with the following observation.

\begin{remark}
\label{rmk: FC centre of odd dihedral Artin group coincides with centre}
    The FC-centre of any dihedral Artin group $G(m)$ coincides with its centre $\mathrm{Z}(G(m))$.
    
    This follows from the exact sequences (\ref{eq: exact sequence for G(m) is central extension of free product}) and (\ref{eveneq: exact sequence for G(m) is central extension of free product}).
    If $m = 2$ then $G(m)$ is $\mathbb{Z}^2$. For any $m\geq 3$, both $\Z/2\Z\ast \Z/m\Z$, for $m$ odd, and $\Z/p\Z\ast \Z$, for $m$ even, are icc (\cite[Example 5.15]{P}). Therefore $G(m)/\mathrm{Z}(G(m))$ is icc in both cases. If $g\notin \mathrm{Z}(G(m))$, its class $\overline{g}$ in $G(m)/\mathrm{Z}(G(m))$ is non-trivial and hence, since $G(m)/\mathrm{Z}(G(m))$ is icc, the conjugacy class of $\overline{g}$ in this quotient is infinite. It follows that the conjugacy class of $g$ is infinite. 
\end{remark}

From the previous remark and \cite[Proposition 3.4]{CP} we immediately get:

\begin{corollary}
    Positive universal formulae are concise in dihedral Artin groups.
\end{corollary}

Restricting to existential formulae, we can show that these are concise in odd dihedral Artin groups and, more generally, in torus knot groups, where, for any two coprime integers $n, m > 1$, the torus knot group $T(n,m)$ is given by the presentation
\begin{equation}
    \label{eq: presentation of torus knot group}
    T(n,m):= \langle x,y \; | \; x^{n} = y^{m} \rangle.
\end{equation}
The following is a special case of more general embedding results in \cite{CGL}.

\begin{proposition}
\label{fact: odd dihedral AG embeds in direct product}

 The torus knot group $T(n,m)$ embeds as a subgroup of finite index into $\Z\times (\Z/n\Z\ast\Z/m\Z)$ via the following map $\iota$.
 
 Let $z$ be the generator of $\Z$ and let $\Z/n\Z\ast \Z/m\Z=\langle a, b\mid a^n=b^m=1\rangle$. Using the presentation (\ref{eq: presentation of torus knot group}), define $\iota$ as follows: 
 \begin{align*}
     \iota(x)&:=(z^m, a)\\
     \iota(y)&:=(z^n, b).
 \end{align*}
\end{proposition}
 In particular, since the centre of $T(n,m)$ is generated by $x^n$, we can identify $\mathrm{Z}(T(n,m))$ inside $\Z\times(\Z/n\Z\ast \Z/m\Z)$ with the subgroup generated by $(z^{nm},1)$.

\torusknotgroups*

\begin{proof}
    Let $G:=T(n,m)$ be a torus knot group, $\varphi$ be an existential formula, and let $g$ be an element of $G_\varphi$ of infinite order.
    One can argue as in Remark \ref{rmk: FC centre of odd dihedral Artin group coincides with centre} to show that $\mathrm{FC}(G)=\mathrm{Z}(G)$, so we can assume that $g\in \mathrm{Z}(G)$.
    Thanks to Proposition \ref{fact: odd dihedral AG embeds in direct product} we can conclude using Proposition \ref{prop: existential formulae are concise in groups embedding nicely in AxK}.   
\end{proof}

\end{document}